\documentclass[12pt]{amsart}
\usepackage[dvips]{color}
\usepackage{amsmath}
\usepackage{amsxtra}
\usepackage{amscd}
\usepackage{amsthm}
\usepackage{amsfonts}
\usepackage{amssymb}
\usepackage{eucal}
\usepackage{epsfig}
\usepackage{graphics}

\usepackage{tikz}
\usetikzlibrary{matrix,arrows,shapes,calc}
\def\({\left(}
\def\){\right)}

\newcommand{\ga}{\gamma}
\newcommand\Ref{\eqref}

\newcommand{\bea}{\begin{eqnarray}}
\newcommand{\ena}{\end{eqnarray}}
\def\bel{\begin{eqnarray}}
\def\enl{\end{eqnarray}}
\newcommand{\be}{\begin{eqnarray*}}
\newcommand{\en}{\end{eqnarray*}}

\newcommand{\R}{{\mathbb R}}

\newcommand{\C}{{\mathbb C}}
\newcommand{\Z}{{\mathbb Z}}

\newcommand{\tr}{{\rm tr}}

\newcommand{\End}{\mathop{\rm End}}

\newenvironment{tenumerate}{
  \begin{enumerate}
  
  }{\end{enumerate}}
\newcommand{\bi}{\begin{tenumerate}}
\newcommand{\ei}{\end{tenumerate}}
\newcommand{\isoto}[1][]%
{{\mathop{\buildrel{\sim}\over\longrightarrow}\limits_{#1}}}

\def\[{\left[}
\def\]{\right]}
\newcommand{\la}{\lambda}

\newcommand{\al}{\alpha}

\numberwithin{equation}{section}
\newtheorem{thm}{Theorem}[section]
\newtheorem{prop}[thm]{Proposition}
\newtheorem{lem}[thm]{Lemma}

\newtheorem{cor}[thm]{Corollary}

\newcommand{\on}{\operatorname}

\def\bi{\mathbf{i}}

\newcommand{\Res}{\operatorname{Res}}

\begin{document}
\begin{title}[On Conjectures of A. Eremenko and A. Gabrielov]
{On Conjectures of A. Eremenko and A. Gabrielov}
\end{title}
\author{E. Mukhin$\>^*$ and V. Tarasov$\>^\star$}
{\let\thefootnote\relax
\footnotetext{\vskip-.8pt\noindent
$^*$\,Supported in part by NSF grant DMS-0900984\\
{}$\quad$ $^\star$\,Supported in part by NSF grant DMS-0901616}

\begin{abstract}
We study polynomials $p(x)$ satisfying a differential equation of the form $p''-h'p'+Hp=0$, 
where $h=x^3/3+ax$. We prove a conjecture of A. Eremenko and A. Gabrielov. 
\end{abstract}

\maketitle

\section{Introduction}
In this note we consider the equations of the form
\bea\label{quartic}
p''(x)-(x^2+a)p'+H(x)p=0
\ena
which have a polynomial solution. Such equations appear in the study of the elementary eigenfunctions of the Schr\"{o}dinger equation with quartic potential, see \cite{EG}. 

We consider the corresponding local system. 
The cohomology of the
system 
is two-dimensional. Our main result is the
proof of a conjecture of \cite{EG}, which describes the cohomology class 
of the polynomial $p^2(-x)$.

The problem of computing the cohomology classes is formulated in an algebraic setting, see Section \ref{wave sec}, but we use complex-analytic tools to solve it. Our main insight comes from 
the consideration of the bispectral dual equation to \Ref{quartic}, see equation \Ref{dual}.

To prove the wanted equality of two constants we interpret them as values at zero of two {\it a priori} different solutions of equation \Ref{dual} and then show that the two solutions actually
are the same comparing their asymptotics via steepest descent method.

\medskip

The paper is written as follows. In Section \ref{elem sec} we discuss the local system associated to \Ref{quartic}. We proceed to describe an explicit basis in the cohomology in Section \ref{wave sec}. In Section \ref{d sec} we exhibit polynomials which are homologoues to a constant multiple of the first basis element proving in particular Conjecture 1 from \cite{EG}. We discuss the bispectral dual equation in Section \ref{dual sec}. Section \ref{lin alg sec} is devoted to the elementary computation with the characteristic equation of the linear operator corresponding to \Ref{quartic}. Finally, we prove our main results Theorem \ref{c} and Corollary \ref{c cor} in Section \ref{c sec}.

\medskip

{\bf Acknowledgments.} We thank A. Eremenko and A. Gabrielov for explaining  
their work and conjectures. We thank A. Its for helpful discussions.

\section{Elementary remarks}\label{elem sec}
Fix $a\in\C$ and let
\be
h(x)=\frac{x^3}{3}+ax\in\C[x].
\en

Denote by the prime the operator of differentiation 
with respect to variable $x$ and define a linear map 
on rational functions of $x$:
\be
D: \C(x)\to \C(x), \qquad q(x)\mapsto q'(x)+h'(x)q(x).
\en
The map $D$ is inherited from the derivative map 
$\frac{d}{dx}:\ \C(x)e^{h(x)}\to\C(x)e^{h(x)}$.

Let $C\subset\C(x)$ be the image of $D$. 
Let 
\be
R=\{q(x)\in\C(x)\ |\ \Res q(x)e^{h(x)}=0\} 
\en
be the subspace of rational function which have no residues after multiplication by the exponential of $h(x)$. We have $C\subset R$.

For $q_1(x),q_2(x)\in\C(x)$, we write $q_1(x)\sim q_2(x)$ 
if and only if $q_1(x)-q_2(x)\in C$.

Let $\gamma_j(t)$, $j=0,1,2$, be any smooth curves in complex plane such that
\be
\lim_{t\to-\infty}\arg(\gamma_j(t))=\pi(1/3+2j/3), \qquad 
\lim_{t\to\infty}\arg(\gamma_j(t))=\pi(1+2j/3)
\en
and $\lim_{t\to\pm\infty}|\ga_j(t)|=\infty$, see Figure \ref{figure}.

\begin{figure}
\begin{tikzpicture}[baseline=0pt,scale=0.9,line width=1pt]
\draw [->,line width=1.5pt](-4,0) -- (4,0);
\draw [->,line width=1.5pt](0,-4) -- (0,4);
\draw [line width=.7pt](0,0) -- (2,3.5);
\draw [line width=.7pt](0,0) -- (2,-3.5);
\draw [line width=.7pt,dashed,rotate=150](-3.5,0) -- (3.5,0);
\draw [line width=.7pt,dashed,rotate=-150](-3.5,0) -- (3.5,0);
\draw[<-,thick] (-3.5,0.2) to [out=0,in=240] (1.8,3.5);
\draw[->,thick] (-3.5,-0.2) to [out=0,in=120] (1.8,-3.5);
\draw[->,thick] (2.1,-3.2) to [out=120,in=240] (2.1,3.2);
\node at (-3.3,.6) {$\ga_0$};
\node at (1.2,-3.3) {$\ga_1$};
\node [right] at (2,3) {$\ga_2$}; 
\node[below] at (4,0) {$\operatorname{Re}$};
\node[below] at (-.5,4) {$\operatorname{Im}$};
\node[below] at (-.2,0) {0};
\draw [line width=.5pt,dotted](0,3.5) -- (-1,3);
\draw [line width=.3pt, dotted](0,3.3) -- (-1,2.8);
\draw [line width=.5pt,dotted](0,3.1) -- (-1,2.6);
\draw [line width=.3pt, dotted](0,2.9) -- (-1,2.4);
\draw [line width=.5pt,dotted](0,2.7) -- (-1,2.2);
\draw [line width=.3pt, dotted](0,2.5) -- (-1,2.0);
\node at (-0.5,2.6) {$H_1$};
\draw [line width=.5pt,dotted,rotate=120](0,3.5) -- (-1,3);
\draw [line width=.3pt, dotted,rotate=120](0,3.3) -- (-1,2.8);
\draw [line width=.5pt,dotted,rotate=120](0,3.1) -- (-1,2.6);
\draw [line width=.3pt, dotted,rotate=120](0,2.9) -- (-1,2.4);
\draw [line width=.5pt,dotted,rotate=120](0,2.7) -- (-1,2.2);
\draw [line width=.3pt, dotted,rotate=120](0,2.5) -- (-1,2.0);
\node at (-2,-2) {$H_2$};
\draw [line width=.5pt,dotted,rotate=240](0,3.5) -- (-1,3);
\draw [line width=.3pt, dotted,rotate=240](0,3.3) -- (-1,2.8);
\draw [line width=.5pt,dotted,rotate=240](0,3.1) -- (-1,2.6);
\draw [line width=.3pt, dotted,rotate=240](0,2.9) -- (-1,2.4);
\draw [line width=.5pt,dotted,rotate=240](0,2.7) -- (-1,2.2);
\draw [line width=.3pt, dotted,rotate=240](0,2.5) -- (-1,2.0);
\node at (2.5,-1) {$H_0$};
\end{tikzpicture}
\caption{Contours $\ga_j$ and planes $H_j$.}\label{figure}
\end{figure}
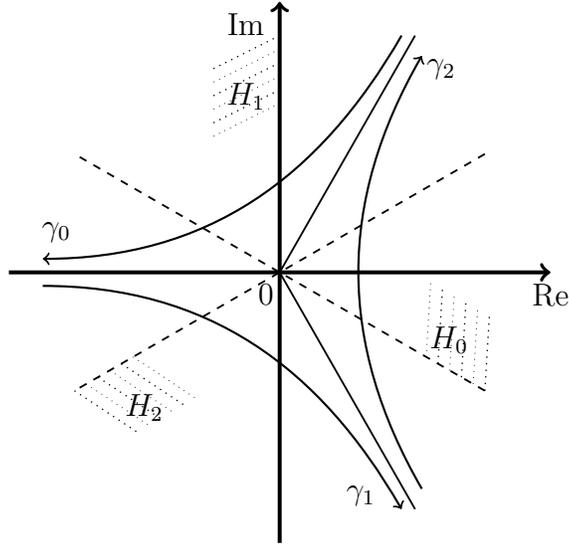

Define functionals $l_j\in R^*$, $j=0,1,2$, by the formula
\be
l_j(q(x))=\int_{\ga_j}q(x)e^{h(x)}dx.
\en
Here  we chose the contour $\ga_j$ so that it does not go through possible poles of $q(x)$.
Clearly, the functionals $l_i$ are well-defined and we have $l_1+l_2+l_3=0$.

\begin{prop}\label{dim2} 
We have $\dim R/C=2$. Moreover, for $q(x)\in R$ we have $q(x)\sim 0$ if and only 
if $l_j(q(x))=0$, $j=1,2$.
\end{prop}
\begin{proof}
Let $q(x)\in R$. 
Write $q(x)$ as a sum of simple fractions. If we have a term $1/(x-z)^k$, for some $z\in\C$, then we subtract 
$D(1/(x-z)^{k-1})$ and decrease the order of the pole modulo $C$. Note, that
since $q(x)\in R$, $k=1$ is impossible. Therefore there exists a polynomial $f(x)$, such that   $q(x)\sim f(x)$. Now if $\deg f(x)=n>1$ we subtract $D(x^{n-2})$ and reduce the degree of $f(x)$ modulo $R$. It follows that $q(x)\sim ax+b$ for some choice $a,b\in\C$.

Since $|e^{h(\ga_j(t))}|\to 0$ as $t\to\pm\infty$, we have $l_j(C)=0$. Therefore to finish the proof of the proposition, it is sufficient to show that $l_1, l_2\in R^*$ are linearly independent functionals. Thus, it is sufficient to show that
$\det (\int_{\gamma_j}x^{k-1}e^{h(x)}dz)_{j,k=1,2}=0$. But this determinant is non-zero because it equals the Wronski determinant $W(\phi_1(u),\phi_2(u))$ where $\phi_j(u)=\int_{\ga_j}e^{h(x)+ux}dx$ are fundamental solutions of the Airy equation $f''(u)+(u+a)f(u)=0$.
\end{proof}
We remark that one can replace the cubic odd polynomial $h(x)$ with an arbitrary polynomial of degree $k$ and similarly
define $k$ functionals and prove a generalization of Proposition \ref{dim2} with $\dim R/C=k-1$.

\medskip 
Let $p(x)\in\C[x]$ be a polynomial of degree $n$ with simple roots only. Let 
\be
R(p(x))=R\cap \frac{\C[x]}{p^2(x)},\qquad C(p(x))= C\cap D\left(\frac{\C[x]}{p(x)}\right).
\en
Then clearly $C(p(x))\subset R(p(x))$. If $q_1(x),q_2(x)\in R(p(x))\subset R$ then clearly $q_1(x)\sim q_2(x)$ if and only if $q_1(x)-q_2(x)\in C(p(x))$.

\begin{lem}\label{pdim2} We have $\dim R(p(x))/C(p(x))=2$. Moreover, for $q(x)\in R(p(x))$ we have 
$q(x)\in R(p(x))$ if and only if $l_i(q(x))=0$, $i=1,2$.
\end{lem}
\begin{proof}
The proof is similar to the proof of Proposition \ref{dim2}.
\end{proof}

\section{When $p(x)$ is a wave polynomial}\label{wave sec}
We start with the following lemma.
\begin{lem}\label{pol lem} Let $p(x)\in\C[x]$ be a polynomial with simple roots only.
We have $1/p^2(x)\in R$ if and only if there exists $b\in \C$ such that $p(x)$ is a solution of the equation
\begin{equation}\label{difur}
y''(x)-h'(x)y'(x)+(nx+b)y(x)=0.
\end{equation}
\end{lem}
\begin{proof} Let $z\in\C$ be such that $p(z)=0$. We compute
\be
\Res_{x=z} \frac{e^{h(x)}}{p^2(x)}=\lim_{x\to z}\left(\frac{e^{h(z)}(x-z)^2}{p^2(z)}\right)'=
e^{h(z)}\frac{h'(z)p'(z)-p''(z)}{(p'(z))^3}
\en
by applying the L'Hopital rule. The lemma follows.
\end{proof} 

We call a polynomial $p(x)$ satisfying \Ref{difur} the {\it wave polynomial of degree $n$}. It  is known that for each $n\in\Z_{\geq 1}$, $a\in\C$, there exists at least one wave polynomial of degree $n$, see also Section \ref{lin alg sec} below.  

Note that all roots of all non-zero wave polynomials are simple.

\begin{lem} Let $p(x)$ be a wave polynomial. Let $f(x)$ be the polynomial such that $f'(x)=p(x)$, $f(0)=0$. Then $f(x)/p^2(x)\in R(p(x))$ and for any $q(x)\in R(p(x))$, there exist unique $c,d\in\C$ such that $q(x)\sim (c+df(x))/p^2(x)$.
\end{lem}
\begin{proof}
The claim that $f(x)/p^2(x)\in R$ is checked similarly to the proof of Lemma \ref{pol lem}. It follows that for any $c,d\in \C$, we have $(c+df(x))/(p(x))^2\in R$.

If $\deg p(x)=n$ then $\deg f(x)=n+1$. 
If $g(x)/p^2(x)\in C$ then $\deg g(x)\geq 
n+2$. The lemma follows from Lemma \ref{pdim2}.
\end{proof}

In what follows we study the constants $c,d$ for $q(x)$ of the form $p(-x)r(x)$, where $r(x)$ is a polynomial of degree at most $n$. In particular, we prove  Conjecture 1 and formula (18) from \cite{EG} describing the constants $c,d$ for $q(x)=p^2(-x)$.

\section{When the constant $d$ is zero}\label{d sec}
Let $p(x)$ be a wave polynomial of degree $n$.

\begin{thm}\label{d=0}
For any polynomial $r(x)$ with $\deg r(x)\leq n$, there exists $c\in \C$ such that
$r(x)p(-x)\sim c/p^2(x)$.
\end{thm}
\begin{proof}
For $j=0,1,2$, consider
\bea\label{sec solution}
y_j(x)=p(x)\int_{\ga_{j,x}}\frac{e^{h(z)}}{p^2(z)} \ dz,
\ena
where the integration is taken over the contour $\ga_{j,x}(t)$, $t\in (-\infty,0]$, such that $\ga_{j,x}(0)=x$ and $\lim_{t\to-\infty}\arg \ga_{j,x}(t)=\pi(1/3+2/3j)$, $\lim_{t\to-\infty}|\ga_{j,x}(t)|=\infty$. Clearly $y_j(x)$ are holomorphic functions satisfying \Ref{difur}. Let $H_j\subset \C$ be the half-planes given by 
\be
H_j=\{z\in\C \ | \ \pi(-1/6+2j/3)<\arg z<\pi(5/6+2j/3)\}, 
\en
see Figure \ref{figure}.
We have the following asymptotics:
\be
y_j(x)=\frac{e^{h(x)}}{x^{n+2}}(1+o(1)), \qquad x\to \infty, \ x \in H_j.
\en
It implies the following connection formulas
\begin{align*}
y_{j+1}(x)&=y_{j}(x)-J_{j}p(x), \qquad j=0,1,2,
\end{align*}
where  
\be
J_j=l_j\left(\frac{1}{p^2(x)}\right)=\int_{\ga_j}\frac{e^{h(x)}}{p^2(x)}\ dx\in \C, \qquad j=0,1,2,
\en
and $y_3(x)=y_0(x)$, $J_3=J_0$.

By Proposition \ref{dim2} and Lemma \ref{pdim2}, it is sufficient to prove 
\be
\left| \begin{matrix} J_1 & J_2 \\  \int_{\ga_1} x^kp(-x)e^{h(x)}\ dx &  \int_{\ga_2} x^kp(-x)e^{h(x)}\ dx \end{matrix}\right| =0, \qquad k=0,1,\dots, n. 
\en
Make the change of variables $x\to -x$ in the integrals and using the connection formulas we obtain that the determinant up to a sign is equal to
\begin{align*}
J_2 \int_{\tilde\ga_1} x^kp(x)e^{-h(x)}\ dx-J_1 \int_{\tilde\ga_2} x^kp(x)e^{-h(x)}\ dx \\
= \int_{\tilde\ga_1}x^k(y_0-y_2)e^{-h(x)}\ dx-\int_{\tilde\ga_2}x^k(y_2-y_1)e^{-h(x)}\ dx \\
=\int_{\tilde\ga_1}x^ky_0e^{-h(x)}\ dx+\int_{\tilde\ga_2}x^ky_1e^{-h(x)}\ dx+\int_{\tilde\ga_0}x^ky_2e^{-h(x)}\ dx.
\end{align*}
Here $\tilde \ga_j$ are the contours given by $\tilde\ga_j(t)=-\ga_j(t)$ for all $t\in\R$.
Note that $\tilde\ga_{j+1}\subset H_{j}$, $j=0,1,2$, and therefore the contour of integration in each of the last three integrals can be sent to infinity inside of $H_j$. It follows that each of the three integrals is zero due to the asymptotics of $y_j(x)$.
\end{proof}

\section{Bispectral dual equation}\label{dual sec}
Motivated by \cite{MTV} we consider the bispectral dual equation to \Ref{difur}:
\bea\label{dual}
u\ddot{g}(u)-n\dot{g}(u)-(u^2-au+b)g(u)=0,
\ena
where the dot denotes the derivative with respect to the variable $u$.

Equation \Ref{dual} is obtained from \Ref{difur} by formal 
replacing the operator of multiplication by $x$ with the operator $d/du$ and the operator $d/dx$ with operator of multiplication by $u$ and placing all derivatives $d/du$ to the right of the operators of multiplication by $u$.

The solution sets of bispectral dual operators are often related by suitable transforms. 
We describe such transforms for the case of bispectral
dual operators \Ref{difur} and \Ref{dual}.

\begin{lem}\label{Fourier1} Let $p(x)$ be a polynomial solution of \Ref{difur}. Then for $j=0,1,2$, the integral
\be
g_j^{[1]}(u)=\int_{\gamma_j}p(-x)e^{h(x)-ux}dx
\en
is well-defined and $g_j^{[1]}(u)$ is a holomorphic solution of \Ref{dual}. 
\end{lem}
\begin{proof} The integral is well-defined since $e^{h(x)}$ is decaying along $\ga_j$. We twice use the integration by parts to compute
\begin{align*}
u\ddot{g}_j^{[1]}-n\dot{g}_j^{[1]}-(u^2-au+b)g_j^{[1]}=\int_{\ga_j}u\left(e^{h-ux}\right)' 
p(-x)-(nx-b) p(-x)dx\\
=\int_{\ga_j}-(p(-x))'e^h\left(e^{-ux}\right)'-(nx-b)p(-x)dx=
\int_{\ga_j}(p(-x)e^{h})'-(nx-b)p(-x)dx=0.
\end{align*}
\end{proof}

\begin{lem}\label{Fourier2} Let $y_j(x)$ be a solution of \Ref{difur} given by \Ref{sec solution}. Then the integral
\be
g_j^{[2]}(u)=u^{n+1}\int_{\gamma_j}y_j(x)e^{-ux}dx
\en
is convergent for $u\in\C$ such that $\on{Re}(ue^{i\pi(1+2j/3)})<0$ and 
$g(u)$ is a solution of \Ref{dual}. 
\end{lem}
\begin{proof}
The integral converges as $t\to -\infty$ on $\ga_j(t)$ since $y_j(x)$ 
is decaying and for $t\to\infty$ since $e^{-ux}$ is decaying.

Similarly to the proof of Lemma \ref{Fourier1}, we use twice the integration by parts and obtain
\begin{align*}
\frac{1}{u^{n+1}}(u\ddot{g}_j^{[2]}-n\dot{g}_j^{[2]}-(u^2-au+\la)g_j^{[2]})=\int_{\ga_j}(uh'-(n+2)x-u^2-\la)y_je^{-ux}dx
\\ =\int_{\ga_j}-(-y_j'+h'y_j)\left(e^{-ux}\right)'-((n+2)x+\la)y_j dx=0.
\end{align*}
\end{proof}

Using integration by parts the function $g^{[2]}_j(u)$ can be rewritten as follows:
\bea\label{rewrite}
{}\\
g_j(u)=u^{n+1}\int_{\ga_j}p(x)e^{-ux}\left(\int_{\infty_j}^x\frac{e^{h(z)}}{p^2(z)}dz\right) dx
=\int_{\ga_j}\sum_{r=0}^nu^{n-r}p^{(r)}(x)\ \frac{e^{h(x)-ux}}{p^2(x)}dx,\notag
\ena
where $p^{(r)}(x)$ denotes the $r$-th derivative of $p(x)$. 
In particular, the integral on the right hand side of \Ref{rewrite} converges 
for all values of $u\in\C$ and the function $g_j^{[2]}(u)$ is holomorphic in $\C$.

\begin{prop}\label{g=g}
For $j=0,1,2$, we  have $g_j^{[1]}(u)=(-1)^ng_j^{[2]}(u)$.
\end{prop}
\begin{proof}
We compute the asymptotics of $g_j^{(i)}$ using the steepest descend method 
similarly to the computation of asymptotics of the Airy functions, see \cite{Ai}. We obtain for $j=0,1,2$, 
\be
g_j^{[1]}(u)=i(-1)^{n+j'}\sqrt{\pi} u^{n/2-1/4}e^{-\frac{2}{3} u^{3/2}+au^{1/2}}(1+o(1))
\en
as $|u|\to \infty$, with $\arg u$ fixed such that
\begin{align*}
\pi/3< &\arg u<7\pi/3 \quad &(j=0), \\
-7\pi/3< &\arg u<-\pi/3 \quad &(j=1), \\
-\pi< &\arg u<\pi \quad &(j=2).
\end{align*}
Here $j'=0$ for $j=2$ and $j'=1$ for $j=0,1$.

Similarly, using \Ref{rewrite}, we conclude that the function $g_j^{[2]}(u)$ 
has the asymptotics different from that of $g_j^{[1]}(u)$ only by a factor of $(-1)^n$ 
and since there is a unique solution of $\Ref{dual}$ with such asymptotics,
the proposition follows. 
\end{proof}

\begin{lem}\label{Fourier-1}
Let $g(u)$ be a solution of \Ref{dual} holomorphic at $u=0$. Then 
\be
p(x)=\Res_{u=0}\frac{g(u)e^{ux}}{u^{n+1}}
\en
is a polynomial solution of \Ref{difur}.
\end{lem}
\begin{proof}
We again use twice the integration by parts:
\begin{align*}
2\pi i\ (p''-h'p'+(nx+b)p)=\int_{|u|=\epsilon} \frac{(u^2-au -x^2u+nx+b)ge^{ux}}{u^{n+1}}\ du \\
=\int_{|u|=\epsilon} \frac{((u^2-au+b)g +n\dot{g}-u\ddot{g})e^{ux}}{u^{n+1}}du=0.
\end{align*}
\end{proof}

\section{Some linear algebra}\label{lin alg sec}
Let $V=\C^{n+1}$ be the vector space with a scalar product. We fix an orthonormal basis $\{e_0,\dots,e_n\}$ in $V$. For $v\in V$ we denote $v_i=v_i\cdot e_i$ the coordinates of $v$.
 For an operator $A\in\End(V)$ we denote $A_{ij}=e_i\cdot Ae_j$ the matrix coefficients of $A$.
We also denote $\widehat{A}$ the adjoint operator of $A$. We have $A\hat A=\hat A A=(\det A) I$.

Let $A: \ V\to V$ be a linear operator with eigenvalue $-b$. Let $v$ and $v^*$ be the corresponding eigenvectors of $A$ and $A^*$. We have $(A+b)v=0$, and $(A^*+b)v^*=0$.
 
\begin{lem}\label{linalg} Assume that $v_jv_k^*\neq 0$. Then
\be
\left(\frac{d}{d\la}\det (A+\la)\right)|_{\la=b}=\frac{v\cdot v^*}{v_jv_k^*} \widehat{(A+b)}_{jk}.
\en
\end{lem}
\begin{proof} Since $-b$ is an eigenvalue, there exists $\al\in\C$ such that
$(\widehat{A+b})_{sl}=\al v_sv_l^*$ for all $s,l=0,1,\dots,n$. Thus
\be
\left(\frac{d}{d\la}\det(A+\la)\right)|_{\la=b}=\tr\ \widehat{A+b}=\al \sum_{s=0}^{n} v_sv_s^*=\frac{\widehat{(A+b)}_{jk}}{v_jv_k^*}\ v\cdot v^*.
\en
\end{proof}

We apply Lemma \ref{linalg} to the case $V=\C_{n}[x]$ 
the space of polynomials of degree at most $n$ and 
\bea\label{A}
A=(d/dx)^2-h'(x)(d/dx)+nx.
\ena
Clearly $A$ is a linear operator which preserves $V$. 
We choose the basis of $V$ as follows: let 
\be 
e_k(x)=x^k/k!, \qquad  k=0,\dots,n.
\en 
We set $e_{-1}(x)=e_{-2}(x)=0$. Then
\bea\label{A1}
Ae_k=e_{k-2}-ae_{k-1}+(n-k)(k+1)e_{k+1},\qquad k=0,\dots, n.
\ena
Clearly, there exists a wave polynomial $p(x)$ 
of operator \Ref{difur} if and only if $-b$ is an eigenvalue of 
$A$. Moreover, in such a case the rank of $D+b$ is $n$, $p(x)$ 
is unique up to a multiplicative constant and the degree of 
$p(x)$ is exactly $n$.

Let $p(x)=\sum_{s=0}^n p_se_s(x)$, $p_s\in\C$, be a wave polynomial: $(A+c)p(x)=0$. 
Then, clearly, $p^*(x)=\sum_{s=0}^np_{n-s}e_s(x)$ satisfies $(A^*+c)p^*(x)=0$.

Using Lemma \ref{linalg} with $j=n, k=0$, we have
\bea\label{der=pp}
\left(\frac{d}{d\la}\det (A+\la)\right)|_{\la=c}=(-1)^n\frac{(n!)^2}{p_n^2}\ \sum_{s=0}^n p_sp_{n-s}.
\ena

\section{The constant $c$}\label{c sec}
We are now ready to compute the constants $c$.
\begin{thm}\label{c} Let $p(x)=\sum_{s=0}^n p_se_s(x)$ be a monic wave polynomial of degree $n$. Then
for $k=0,\dots,n$ we have
\be
e_k(-x)p(-x)\sim \frac{(-1)^np_{n-k}}{p^2(x)}.
\en
\end{thm}
\begin{proof}
By Lemma \ref{pdim2}, there exists a constant $c_k$ such that 
$e_k(-x)p(-x)\sim c_k/p^2(x)$ and for $j=0,1,2$,
\be
J_jc_k=\int_{\ga_j} e_k(-x)p(-x) dx.
\en
Choose any $j\in\{0,1,2\}$ and
set $g_j(u)=g^{[1]}_j(u)=(-1)^ng^{[2]}_j(u)$, see Proposition \ref{g=g}.
Using the presentation $g_j(u)=g^{[1]}_j(u)$, we obtain
\be
J_j c_k= \int_{\ga_j} e_k(-x)p(-x)e^h \ dx=\frac{g_j^{(k)}(0)}{k!},
\en
where $g_j^{(k)}(u)$ denotes the $k$-th derivative of $g_j(u)$.

Expanding $g_j(u)$ in the Taylor series at $u=0$ and using Lemma \ref{Fourier-1} we compute
\be
\alpha p(x)=\Res_{u=0}\sum_{s=0}^{\infty}\frac{g_j^{(s)}(0)u^{s-n-1}}{s!}\  e^{ux}=
\sum_{s=0}^\infty\frac{g_j^{(s)}(0)}{s!}\ e_{n-s}(x),
\en
Since $p(x)$ is monic, we have $p_n=n!$ and the constant $\alpha$ 
is given by $g_j(0)/n!$.
It follows that
\be
c_k=\frac{g_j^{(k)}(0)}{k!J_j}=\frac{g_j(0)}{n!J_j} p_{n-k}.
\en
Finally, using that $g_j(u)=(-1)^ng_j^{[2]}(u)$ and equation 
\Ref{rewrite}, we obtain
\be 
g_j(0)=(-1)^n n!J_j.
\en
The theorem follows.
\end{proof}

\begin{cor}\label{c cor} Let $p(x)$ be a monic wave polynomial of degree $n$. Then
\be
p^2(-x)\sim \frac{c}{p^2(x)}, \qquad c=\left(\frac{d}{d\la}\det (A+\la)\right)|_{\la=b},
\en
where $A$ is given by \Ref{A} or \Ref{A1}.
\end{cor}
\begin{proof}
By Theorem \ref{c}, we have
\be
p^2(-x)=\left(\sum_{s=0}^n p_se_s(-x)\right)p(-x)\sim (-1)^n\frac{\sum_{s=0}^n p_sp_{n-s}}{p^2(x)}.
\en
The corollary now follows from \Ref{der=pp}.
\end{proof}

\begin{cor}
Conjecture 1 and formula (18) in \cite{EG} is true. $\qquad$ $\Box$
\end{cor}
\begin{proof}
Conjecture 1 and formula (18) in \cite{EG} follow from Theorem \ref{d=0} and
Corollary \ref{c cor} respectively after the change of variables:
\bea\notag
x=\beta z_{EG},\qquad p(x)=p_{EG}(z)\beta^n, \qquad \beta a= 2b_{EG}, 
\qquad \beta^2 b=2a_{EG},
\ena
where $\beta^3=-2$ and we denote the objects from \cite{EG} 
by the same letters as there 

but with the index $EG$ to distinguish from the notation used in this note.
\end{proof}

\bigskip

\address{EM: {\it Department of Mathematical Sciences,
Indiana University\,--\,Purdue University,
Indianapolis, 402 North Blackford St, Indianapolis,
IN 46202-3216, USA}}

\email{mukhin@math.iupui.edu}

\medskip

\address{VT: {\it Department of Mathematical Sciences,
Indiana University\,--\,Purdue University,
Indianapolis, 402 North Blackford St, Indianapolis,
IN 46202-3216, USA, and St.\,Petersburg Branch of Steklov Mathematical
Institute Fontanka 27, St.\,Petersburg, 191023, Russia}}

\email{vt@math.iupui.edu}

\email{vt@pdmi.ras.ru}
\end{document}